\documentclass[11pt,oneside]{amsart}%
\usepackage{amssymb,amsmath,amsfonts,latexsym,amsthm,geometry,graphicx}
\usepackage[normalem]{ulem}
\usepackage{amsmath}
\usepackage{amsfonts}
\usepackage{color, soul}
\usepackage{amssymb}
\usepackage{tikz-cd}
\usepackage[all]{xy}%
\providecommand{\U}[1]{\protect\rule{.1in}{.1in}}
\newtheorem{theorem}{Theorem}[section]
\newtheorem{corollary}[theorem]{Corollary}
\newtheorem{proposition}[theorem]{Proposition}

\theoremstyle{definition}

\newtheorem{remark}[theorem]{Remark}
\newtheorem{definition}[theorem]{Definition}
\begin{document}
	\title{A unified factorization theorem for Lipschitz summing operators}
	\author[Geraldo Botelho]{Geraldo Botelho}
	\address{Faculdade de Matem\'{a}tica\\
		Universidade Federal de Uberl\^{a}ndia\\
		38.400-902, Uberl\^{a}ndia, Brazil.}
	\email{botelho@ufu.br}
	\author[Mariana Maia]{Mariana Maia}
	\address{Departamento de Ci\^encia e Tecnologia\\
		Universidade Federal Rural do Semi-\'{A}rido\\
		59.700-000 - Cara\'ubas, Brazil.}
	\email{mariana.britomaia@gmail.com or mariana.maia@ufersa.edu.br}
	\author[Daniel Pellegrino]{Daniel Pellegrino}
	\address{Departamento de Matem\'{a}tica \\
		Universidade Federal da Para\'{\i}ba \\
		58.051-900 - Jo\~{a}o Pessoa, Brazil.}
	\email{dmpellegrino@gmail.com}
	\author[Joedson Santos]{Joedson Santos}
	\address{Departamento de Matem\'{a}tica \\
		Universidade Federal da Para\'{\i}ba \\
		58.051-900 - Jo\~{a}o Pessoa, Brazil.}
	\email{joedsonmat@gmail.com or joedson@mat.ufpb.br}
	\thanks{Geraldo Botelho is supported by FAPEMIG and CNPq, Daniel Pellegrino is supported by CNPq and Joedson Santos is supported by CNPq}
	\maketitle
	
	\begin{abstract}
		We prove a general factorization theorem for Lipschitz summing operators in the context of metric spaces which recovers several linear and nonlinear factorization theorems that have been proved recently in different environments. New applications are also given. 		
	\end{abstract}


	\section{Introduction}
	
	\bigskip The modern theory of absolutely summing operators, which goes far beyond the original linear theory, is a consequence of
	ideas that go back to pioneer works of Grothendieck, Pietsch, Mitiagin,
	Lindenstrauss and Pelzczynski (see \cite{G, LP, MP}). More than abstract
	results, the theory provides machinery to deal with important issues of Banach
	Space Theory. For instance, in the classical paper of Lindenstrauss and
	Pelczynski \cite{LP}, they show, as applications of the theory, the following highly
	nontrivial result: all normalized unconditional basis
	of $\ell_{1}\left(  \Gamma\right)  $ are equivalent to the canonical basis.
	
The purpose of this paper is to provide a unified approach, in the linear and nonlinear settings, for one of the most important aspects of theory, namely, the validity of a Pietsch-type factorization theorem in several classes of summing operators between metric and Banach spaces.

Of course, everything started with the classical linear Pietsch Factorization Theorem, which we recall now. Henceforth, $E,F$ are Banach spaces over $\mathbb{K}=\mathbb{R}$ or
	$\mathbb{C}$ and $B_{E^{\ast}}$ denotes the closed unit ball of the
	topological dual $E^{\ast}$ of $E.$ For $1\leq p<\infty$, we say that a
	linear operator $u\colon E\longrightarrow F$ is absolutely
	$p$-summing (or $p$-summing) if there is a constant $C\geq0$ such that
	\begin{equation*}
	\left(  \sum_{j=1}^{m}\Vert u(x_{j})\Vert^{p}\right)  ^{\frac{1}{p}}\leq
	C\sup_{x^{\ast}\in B_{E^{\ast}}}\left(  \sum_{j=1}^{m}|x^{\ast}(x_{j}%
	)|^{p}\right)  ^{\frac{1}{p}}, \label{A}%
	\end{equation*}
	for all $x_{1},\ldots,x_{m}\in E$ and $m\in\mathbb{N}$.
	
Part of the striking success of this class of operators is due to the following  characterizations, known as the Pietsch Domination Theorem (PDT) and the Pietsch Factorization Theorem (PFT): a linear operator $u\colon E\longrightarrow F$ is absolutely $p$-summing if and only if:\\
$\bullet$  There exist a constant $C$ and a regular Borel probability measure $\mu$ on $B_{E^{\ast}}$ with the weak* topology such that
\begin{equation*}
	\left\Vert u(x)\right\Vert \leq C\left(  \int\nolimits_{B_{E^{\ast}}%
	}\left\vert \varphi(x)\right\vert ^{p}d\mu\right)  ^{1/p} {\rm ~for~every~} x \in E. \label{BBB}%
	\end{equation*}
$\bullet$ There exist a regular Borel probability measure $\mu$ on $B_{E^{\ast}}$ with the weak* topology and a bounded linear operator $B \colon L_p(B_{E^{\ast}}, \mu) \longrightarrow \ell_\infty(B_{F^*})$ such that %
the following diagram is commutative
	\begin{equation*}
	\begin{tikzcd} C (B_{E^*}) \arrow{rr}{j_p} & & L_p (B_{E^{\ast}},\mu) \arrow{d}{B} \\ E \arrow[swap]{u}{i_{E}} \arrow[swap]{r}{u} & F \arrow[swap]{r}{i_{F}} & \ell_{\infty}\left( B_{F^*}\right) \end{tikzcd} \label{CCC}%
	\end{equation*}
	where $j_p$ is the formal inclusion and $i_E$ is the canonical linear embedding, that is, $i_E(x)(x^*) = x^*(x)$ for $x \in E$ and $x^* \in B_{E^*}$. As $B_{E^{\ast}}$ is a weak* compact set, $i_{E}$ takes its values in $C(B_{E^{\ast}})$.

	Naturally enough, the characterizations above lie at the heart of
	the generalizations of the class of $p$-summing operators  pursued by different
	authors in several recent papers. It is worth mentioning that nowadays
	absolutely summing operators are mostly investigated in the nonlinear setting, mainly for multilinear and polynomial operators between Banach spaces and Lipschitz maps between metric spaces. As a result, many PDTs and PFTs have been obtained for different classes of linear and nonlinear summing operators. Following
	this trend, a series of papers (\cite{BPRn, psjmaa, adv}) have investigated in
	depth how far the PDT holds in the nonlinear setting. Ultimately, it has been definitively proved in \cite{adv} that
	the PDT holds in an extremely relaxed environment, with almost no
	structure needed. Other properties of summing operators, such as extrapolation	type theorems, also hold in a very abstract setting (see \cite{LONDON}).
	However, the PFT seems to be more restrictive and a
	result as general as those from \cite{adv, LONDON} is still not
	available. The aim of this paper is to fill this gap by providing a general version of the PFT	that recovers, as particular cases, several factorization theorems for classes of summing linear and nonlinear operators proved thus far by different authors. Paraphrasing
	\cite{mexicana}, our purpose is to show that the
	\textquotedblleft triad\textquotedblright
\begin{center} Summability property $\Leftrightarrow$ Domination Theorem $\Leftrightarrow$ Factorization Theorem \end{center}
holds at a very high level of generality.
%
\section{Results}

Throughout this section, $X$ is an arbitrary non-void set, $Y$ is a metric
space, $K$ is a compact Hausdorff space, $C(K)=C(K;\mathbb{K})$
is the space of all continuous $\mathbb{K}$-valued functions with the $\sup$
norm ($\mathbb{K}=\mathbb{R}$ or $\mathbb{C}$), $\Psi\colon
X\longrightarrow C(K)$ is an arbitrary map and $p \in [1, \infty)$. By $d_Y$ we denote the metric on $Y$.

\begin{definition}
A map $u \colon X\longrightarrow Y$ is said to be \textbf{$\Psi$-Lipschitz
$p$-summing} if there is a constant $C\geq0$ such that
\begin{equation*}
\left(  \sum_{j=1}^{m}d_{Y}(u(x_{j}),u(q_{j}))^{p}\right)  ^{\frac{1}{p}}\leq
C\sup_{\varphi\in K}\left(  \sum_{j=1}^{m}\left\vert \Psi(x_{j})(\varphi
)-\Psi(q_{j})(\varphi)\right\vert ^{p}\right)  ^{\frac{1}{p}}, \label{34M}%
\end{equation*}
for all $x_{1},\ldots,x_{m},q_{1},\ldots,q_{m}\in X$
and $m\in\mathbb{N}$.
\end{definition}

Given a regular Borel probability $\mu$ on $K$, by $j_{p} \colon C(K)\longrightarrow L_{p}(K,\mu)$ we denote the canonical operator. Now consider the map
\[
j_{\mu,p}^{\Psi}\colon X\longrightarrow L_{p}(K,\mu)~,~j_{\mu,p}^{\Psi}%
:=j_{p}\circ\Psi.
\]
Note that, for every $x \in X$,
$$j_{\mu,p}^{\Psi}\left(  x\right)  =\left(  j_{p}%
\circ\Psi\right)  \left(  x\right)  =j_{p}\left(  \Psi(x)\right)  =\Psi(x),$$
so
\begin{equation}
\label{jf}\left\Vert j_{\mu,p}^{\Psi}(x)- j_{\mu,p}^{\Psi} (q)\right\Vert
_{L_{p}(K,\mu)} = \left\Vert j_{p}\circ\Psi(x)- j_{p}\circ\Psi(q)\right\Vert
_{L_{p}(K,\mu)} = \left\Vert \Psi(x)-\Psi(q)\right\Vert _{L_{p}(K,\mu)}.
\end{equation}

\begin{remark}
\label{obs} Alternatively, one can consider the canonical operator $I_{\infty
,p}\colon L_{\infty}(K,\mu)\longrightarrow L_{p}(K,\mu)$. In this case the operator
$I_{\infty,p}\circ j_{\mu,\infty}^{\Psi}\colon X\longrightarrow L_{p}(K,\mu)$
satisfies $I_{\infty,p}\circ j_{\mu,\infty}^{\Psi}(x)=\Psi(x)$ for every $x$ and
\begin{align*}
\left\Vert I_{\infty,p}\circ j_{\mu,\infty}^{\Psi}(x)-I_{\infty,p}\circ
j_{\mu,\infty}^{\Psi}(q)\right\Vert _{L_{p}(K,\mu)}  &  =\left\Vert
\Psi(x)-\Psi(q)\right\Vert _{L_{p}(K,\mu)}\\
&  =\left(  {\displaystyle\int\nolimits_{K}}\left\vert \Psi(x)(\varphi
)-\Psi(q)(\varphi)\right\vert ^{p}d\mu(\varphi)\right)  ^{1/p}.
\end{align*}
\end{remark}

\bigskip To state our main result we first recall the concept of Lipschitz
retraction (see \cite[Proposition 1.2]{BL}). Let $X$ be a subset of the metric space $Y$. A
Lipschitz map $r\colon Y\longrightarrow X$ is called a {\it Lipschitz retraction} if its restriction to $X$ is the
identity on $X$. When such a Lipschitz retraction exists, $X$ is said to be a
{\it Lipschitz retract of $Y$}. A metric space $X$ is called an {\it absolute Lipschitz
retract} if it is a Lipschitz retract of every metric space containing it.

According to \cite{BL}, Lipschitz retractions for metric spaces $X$ are
characterized by the following equivalences:

\begin{enumerate}
\item[(i)] $X$ is an absolute Lipschitz retract.

\item[(ii)] For every metric space $Y$ and for every subset $Z\subset Y$,
every Lipschitz function $f\colon Z\longrightarrow X$ can be extended to a Lipschitz
function $F\colon Y\longrightarrow X$.

\item[(iii)] For every metric space $Y$ containing $X$ and for every metric
space $Z$, every Lipschitz function $f\colon X\longrightarrow Z$ can be extended to a
Lipschitz function $F\colon Y\longrightarrow Z$.
\end{enumerate}


For instance, for every set $\Gamma$, $\ell_{\infty}\left(  \Gamma\right)  $ is an absolute Lipschitz
retract (see \cite[Lemma 1.1]{BL}), and it is also known that any
$C(K;\mathbb{R})$ is an absolute Lipschitz retract (see \cite[Theorem 1.6]{BL}
and \cite[Theorem 6(b)]{Lindenstrauss}). Now we are able to state and prove
our main result:

\begin{theorem}
\label{tf} Let $1\leq p < \infty$, $X$ be an arbitrary non-void set, $Y$ be a
metric space, $K$ be a compact Hausdorff space and $\Psi\colon
X\longrightarrow C(K)$ be an arbitrary map. The following assertions are equivalent for a map $u$ from $X$ to $Y$.

\textbf{(a)} $u$ is $\Psi$-Lipschitz $p$-summing.

\textbf{(b)} There is a regular Borel probability measure $\mu$ on $K$ and a
constant $C\geq0$ such that
\[
\label{npdt}d_{Y}(u(x),u(q)) \leq C\left(  {\displaystyle\int\nolimits_{K}}
\left\vert \Psi(x)(\varphi)-\Psi(q)(\varphi) \right\vert ^{p}d\mu
(\varphi)\right)  ^{1/p}%
\]
for all $ x,q  \in X$.

\textbf{(c)}  There is a regular Borel probability measure $\mu$ on $K$, a closed subset $X_p$ of $L_{p} (K,\mu)$ and a Lipschitz map $\hat{b}: X_p \longrightarrow Y$ such that $(i)\ j_{p}\left(  \Psi(X)\right) \subset X_p$ and $(ii)\ \hat{b}j_{p}\Psi(x)=u(x)$ for all $x\in X$. In other words, the following diagram commutes:	
	\[
	\begin{tikzcd}	
	C(K) \arrow{r}{j_p}\ &\ L_p (K,\mu)
	\end{tikzcd}
	\]\vspace{-0.8cm}
	\[
	\begin{tikzcd}
	\rotatebox{90}{$\subseteq$} & & \hspace{0.5cm} \rotatebox[origin=c]{90}{$\subseteq$}	\end{tikzcd}
	\]\vspace{-0.8cm}
	\[
	\begin{tikzcd}
	\Psi(X)  \arrow{rr}{j_p\mid_{\Psi(X)}} & & X_p \arrow{d}{\hat{b}} \\
	X  \arrow[swap]{u}{\Psi} \arrow[swap]{rr}{u}  & & Y
	\end{tikzcd}
	\]

\textbf{(d)}  There is a regular Borel probability measure $\mu$ on $K$ such that for some (or any) isometric embedding $J$ of $Y$ into an absolute
Lipschitz retract space $Z$, there is a Lipschitz map $B \colon L_{p} (K,\mu)\longrightarrow
Z$ such that the following diagram commutes
\[
\begin{tikzcd}
	C (K) \arrow{rr}{j_p}  & & L_p (K,\mu) \arrow{d}{B} \\
	X \arrow[swap]{u}{\Psi} \arrow[swap]{r}{u} & Y \arrow[swap]{r}{J} & Z
	\end{tikzcd}
\]

\textbf{(e)} There is a regular Borel probability measure $\mu$ on $K$ such that for some (or any) isometric embedding $J$ of $Y$ into a absolute
Lipschitz retract space $Z$, there is a Lipschitz map $B\colon L_{p} (K,\mu)\longrightarrow
Z$ such that the following diagram commutes
\[
\begin{tikzcd}
	L_\infty (K,\mu) \arrow{rr}{I_{\infty,p}}  & & L_p (K,\mu) \arrow{d}{B} \\
	X \arrow[swap]{u}{j_{\mu,\infty}^{\Psi}} \arrow[swap]{r}{u} & Y \arrow[swap]{r}{J} & Z
	\end{tikzcd}
\]
\end{theorem}

\begin{proof}
\textbf{(a)}$\Leftrightarrow$\textbf{(b)} Using the abstract framework introduced in
\cite{psjmaa}, as well as its notation, it is easy to check that that the set of $\Psi$-Lipschitz $p$-summing operators from $X$ to $Y$ is contained in the class of $RS$-abstract $p$-summing mappings for the following choices: $\mathcal{H}$ is the  family of all functions from $X$ to $Y$, $E=X\times X$, $G = \mathbb{R}$,
\[
S\colon\mathcal{H}\left(  X;Y\right)  \times(X\times X)\times\mathbb{R}%
\longrightarrow\lbrack0,\infty)~,~S(u,(x,q),\lambda)=d_{Y}(u(x),u(q)),
\]
and
\[
R\colon K\times(X\times X)\times\mathbb{R}\longrightarrow\lbrack
0,\infty)~,~R(\varphi,(x,q),\lambda)= \left\vert \Psi(x)(\varphi)-\Psi
(q)(\varphi)\right\vert .
\]
As $R_{(x,q),\lambda}(\cdot):=R\left(  \cdot,(x,q),\lambda\right)  $ is continuous for all
$(x,q)\in X\times X$ and $\lambda\in\mathbb{R}$, calling on \cite[Theorem
3.1]{psjmaa} we have that $u\colon X\longrightarrow Y$ is $\Psi
$-Lipschitz $p$-summing if and only if there is a regular Borel probability
measure $\mu$ on $K$ and a constant $C\geq0$ such that
\[
\label{mpdt}d_{Y}(u(x),u(q)) \leq C\left(  {\displaystyle\int\nolimits_{K}}
\left\vert \Psi(x)(\varphi)-\Psi(q)(\varphi) \right\vert ^{p}d\mu
(\varphi)\right)  ^{1/p}%
\]
for all $ x,q  \in X$.


\textbf{(b)}$\Rightarrow$\textbf{(c)} By \textbf{(b)} there exist a regular Borel probability measure $\mu$ on $K$ and a constant $C\geq0$ such that
\begin{equation*}
\label{aplic}d_{Y}(u(x),u(q)) \leq C\cdot\left\Vert
\Psi(x)-\Psi(q)\right\Vert _{L_{p}(K,\mu)}%
\end{equation*}
for all $ x,q  \in X$.

Define $b \colon j_{p}\circ\Psi \left(  X\right)
\longrightarrow Y$ by $b\left(  j_{p}\circ\Psi\left(  x\right)
\right)  =u \left(  x\right)$ and let us that it is well defined.
If $x,q \in X$ are such that $j_{p}\circ\Psi\left(  x\right)  =j_{p}\circ\Psi\left(  q\right),
$ then
\begin{align*}
d_{Y}\left(  b\left(  j_{p}\circ\Psi\left(  x\right)  \right)  ,b\left(
j_{p}\circ\Psi\left(  q\right)  \right)  \right)   &   = d_{Y}\left(  u(x),u(q)\right) \\
&  \overset{\mathbf{(b)}}{ \leq} C\left\Vert
\Psi(x)-\Psi(q)\right\Vert _{L_{p}(K,\mu)}\\
&  \overset{(\ref{jf})}{=} C\left\Vert j_{p}\circ\Psi(x)-j_{p}\circ\Psi(q)\right\Vert _{L_{p}%
	(K,\mu)}=0.
\end{align*}
Therefore $b\left(  j_{p}\circ\Psi\left(  x\right)  \right)  =b\left(j_{p}\circ\Psi\left(  q\right)  \right)  $ and $b$ is Lipschitz.

Considering $X_p$ the norm closure of $j_{p}\circ\Psi\left(  X\right)$ in ${L_{p}(K,\mu)}$ and $\hat{b}: X_p \longrightarrow Y$ the natural extension of $b$ to $X_p$, it follows that $\hat{b}$ is a Lipschitz map and $\hat{b}\circ j_{p}\circ\Psi=u$.

\textbf{(c)}$\Rightarrow$\textbf{(d)} Let $J\colon Y \longrightarrow Z$ be an isometric embedding. Since $Z$ is an absolute Lipschitz retract, it follows that $J\circ \hat{b}$ has a
Lipschitz extension $B \colon L_{p} (K,\mu)\longrightarrow Z$ such that $B\circ
j_{p}\circ\Psi\left(  x\right)  =J\circ u \left(  x\right)  $ for every $x\in X$.

\textbf{(d)}$\Rightarrow$\textbf{(e)} This implication is obvious.

\textbf{(e)}$\Rightarrow$\textbf{(b)} Using that $B$ is Lipschitz, $B\circ
I_{\infty,p}\circ j_{\mu,\infty}^{\Psi}\left(  x\right)  =J\circ u\left(
x\right)  $ for every $x\in X$ and Remark \ref{obs}, we get
\begin{align*}
d_{Y}\left(  u(x),u(q)\right)   &  =d_{Z}\left(  J\circ u\left(  x\right)
,J\circ u\left(  q\right)  \right) \\
&  =d_{Z}\left(  B\circ I_{\infty,p}\circ j_{\mu,\infty}^{\Psi}\left(
x\right)  ,B\circ I_{\infty,p}\circ j_{\mu,\infty}^{\Psi}\left(  q\right)
\right) \\
&  \leq C\cdot\left\Vert I_{\infty,p}\circ j_{\mu,\infty}^{\Psi}%
(x)-I_{\infty,p}\circ j_{\mu,\infty}^{\Psi}(q)\right\Vert _{L_{p}(K,\mu)}\\
&  =C\cdot\left(  {\displaystyle\int\nolimits_{K}%
}\left\vert \Psi(x)(\varphi)-\Psi(q)(\varphi)\right\vert ^{p}d\mu
(\varphi)\right)  ^{1/p}%
\end{align*}
for all $  x,q \in X$.
\end{proof}

\begin{remark}\label{nnrem} Since the map $B$ above is defined on the Banach space $L_p(K,\mu)$, a glance at the classical Pietsch Factorization Theorem makes the following question quite natural: if, in Theorem \ref{tf}, $Y$ and $Z$ are Banach spaces and $J$ is a linear embedding (metric injection), can the map $B$ be chosen to be a linear operator? In the next section we shall see that this is not the case, showing that Theorem \ref{tf} cannot be improved in this direction.
\end{remark}

\begin{corollary}\label{simple}
Let $X, Y, K, \Psi$ and $u$ be as in Theorem \ref{tf}. If there exist a regular Borel probability measure $\mu$ on $K$ and a
Lipschitz map $\widetilde{u}\colon L_{p}(K,\mu)\longrightarrow Y$ such that the the
diagram
\begin{equation*}
\begin{tikzcd}
	C(K) \arrow{r}{j_{p}}  & L_p (K,\mu) \arrow{d}{\widetilde{u}} \\
	X \arrow[swap]{u}{\Psi} \arrow[swap]{r}{u} & Y
	\end{tikzcd}
\end{equation*} commutes, then $u$ is $\Psi$-Lipschitz $p$-summing.
\end{corollary}

\begin{proof} Let the measure $\mu$ and the Lipschitz map $\widetilde{u}$ be as in the assumption. By \cite[Lemma 1.1]{BL} there exist a set $\Gamma$ and an embedding $J$ from $Y$ into the absolute Lipschitz retract $\ell_{\infty}(\Gamma)$. 
Defining
$$B \colon L_{p}(K,\mu)\longrightarrow\ell_{\infty}(\Gamma)~,~B=J\circ\tilde{u},$$ it follows that $B$ is a Lipschitz map and
$$(J\circ u)(x)=(J\circ\tilde{u}\circ
j_{2}\circ\Psi)(x)=(B\circ j_{2}\circ\Psi)(x)$$
for every $x\in X$, that is, the
following diagram commutes
\[
\begin{tikzcd}
	C(K) \arrow{rr}{j_{p}}  & & L_p (K,\mu) \arrow{d}{B} \\
	X \arrow[swap]{u}{\Psi} \arrow[swap]{r}{u} & Y \arrow[swap]{r}{J} & \ell_{\infty}(\Gamma).
	\end{tikzcd}
\]
The conclusion follows from Theorem \ref{tf}.
\end{proof}

The converse of the corollary above holds for absolutely 2-summing {\it linear} operators (see \cite[Corollary 2.16]{djt}). We do not know if the same holds in the nonlinear setting. It is not difficult to check that, if $u$ is $\Psi$-Lipschitz $2$-summing and, in addition, $Y$ is complete and in Theorem \ref{tf}(d), $Z$ can be supposed to be a linear space, $B$ to be linear and $J(Y)$ to be a subspace of $Z$, then there exists a
Lipschitz map $\widetilde{u}\colon L_{2}(K,\mu)\longrightarrow Y$ such that the diagram
\begin{equation*}
\begin{tikzcd}
	C(K) \arrow{r}{j_{2}}  & L_2 (K,\mu) \arrow{d}{\widetilde{u}} \\
	X \arrow[swap]{u}{\Psi} \arrow[swap]{r}{u} & Y
	\end{tikzcd}
\end{equation*}
commutes. 
But, as announced in Remark \ref{nnrem}, we will prove in the next section that the map $B$ cannot be supposed to be linear. So, we have the:

\medskip

\noindent{\bf Open problem.} Does the converse of Corollary \ref{simple} hold for $p=2$?

\section{Applications}

Among other applications, in this section we show that Theorem \ref{tf} recovers, as particular instances, several theorems proved separately in the literature.  

\begin{itemize}
\item {\bf Absolutely summing linear operators.}

Let $1\leq p<\infty$, $X,Y$ be Banach spaces and $u \colon X\longrightarrow Y$ be a bounded linear operator. Letting $
K=B_{X^{\ast}}$ endowed with the weak* topology and
\[
\Psi\colon X\longrightarrow C(B_{X^{\ast}})~, \Psi(x)(x^{\ast})=x^{\ast
}(x),
\]
$u$ is $\Psi$-Lipschitz $p$-summing if and only if there is a constant
$C\geq0$ such that
 \begin{equation*}
 \left(  \sum_{j=1}^{m}\Vert u(x_{j}- q_{j})\Vert_{Y}^{p}\right)  ^{\frac{1}{p}}\leq
C\sup_{x^{\ast}\in B_{X^{\ast}}}\left(  \sum_{j=1}^{m}|x^{\ast}(x_{j}
 - q_{j}
)|^{p}\right)  ^{\frac{1}{p}},\label{36M}
\end{equation*}
for all $x_{1},\ldots ,x_{m},q_{1},\ldots,q_{m} \in X$ and $m\in\mathbb{N}$.
Thus $u$ absolutely $p$-summing if and only if $u$ is $\Psi
$-Lipschitz $p$-summing. Applying condition (c) of Theorem \ref{tf} for $Z = \ell_\infty(B_{F^*})$ and the canonical embedding $i_F \colon F \longrightarrow \ell_\infty(B_{F^*})$, the linearity of $u$ and $\Psi$
 assure that the map $b$ of the proof of the theorem is linear as well. So, the injectivity of $\ell_{\infty}(B_{Y^*})$ allows us to choose $B$ to be a bounded linear operator. In this fashion, Theorem \ref{tf} recovers the classical Pietsch Factorization
Theorem for absolutely $p$-summing linear operators \cite[Theorem 2.13]{djt}.
\end{itemize}

\medskip

\begin{itemize}

\item {\bf Lipschitz $p$-summing operators.}

Let $1 \le p < \infty$, $X$ be a pointed metric space with distinguished point $0$, $X^{\#}$ be the space of all real valued Lipschitz functions on $X$ vanishing at $0$ endowed with the Lipschitz norm (see, e.g. \cite{FJ, saadi}), and $Lip(X;Y)$ be the set of all
Lipschitz maps from $X$ to the metric space $Y$. As $K=B_{X^\#}$ is compact Hausdorff with the topology of pointwise convergence (or, alternatively, as $X^\#$ is a dual Banach space, $K =B_{X^{\#}}$ is compact Hausdorff with the weak* topology), we can consider our construction associated to the map
\[
\Psi\colon X \longrightarrow C(B_{X^{\#}};\mathbb{R})~,~ \Psi(x)(f) = f(x).
\]
So, for map $u\in Lip(X;Y)$, we have that $u$ is $\Psi$-Lipschitz $p$-summing if and only if 
$u$ is Lipschitz $p$-summing in the sense of \cite{FJ}, and, in this case, 
Theorem \ref{tf} recovers the
corresponding factorization theorem \cite[Theorem 1]{FJ}.
\end{itemize}

Now we are ready to answer the question posed in Remark \ref{nnrem}.

\begin{proposition} Suppose that, in Theorem \ref{tf}, $Y$ and $Z$ are Banach spaces and $J$ is a linear embedding. In general, the map $B$ that closes the commutative diagram cannot be chosen to be a linear operator.
\end{proposition}

\begin{proof} Assume that, under the prescribed conditions, $B$ can  always be chosen to be a linear operator. Let $T \colon X \longrightarrow Y$ be a Lipschitz $p$-summing operator, $1 \leq p < \infty$, from a pointed metric space $X$ to a Banach space $Y$. Denote by $\mathcal{ F}(X)$ the free Banach space associated to $X$, by $\delta_X \colon X \longrightarrow \mathcal{ F}(X)$ the canonical embedding and by $T_L \colon \mathcal{ F}(X) \longrightarrow Y$ the linearization of $T$, that is, $T_L$ is linear, bounded and $T  = T_L \circ \delta_X$. Remembering that $X^{\#} = \mathcal{F}(X)^*$ isometrically, the Pietsch Factorization Theorem for this class of operators, which we have just seen above, gives a measure $\mu$ on $B_{X^{\#}} = B_{\mathcal{F}(X)^*}$ and a Lipschitz map $B$ such that the following diagram is commutative:
	
\begin{equation*}
	\begin{tikzcd} C (B_{\mathcal{F}(X)^*}) \arrow{rr}{j_p} & & L_p (B_{\mathcal{F}(X)^*},\mu) \arrow{d}{B} \\ X \arrow[swap]{u}{i_{\mathcal{F}(X)} \circ \delta_X} \arrow[swap]{r}{T} & Y \arrow[swap]{r}{i_{Y}} & \ell_{\infty}\left( B_{Y^*}\right) \end{tikzcd} \label{CCC}%
	\end{equation*}
(where $i_Y$ is the canonical embedding). Our assumption says that $B$ can be supposed to be linear (we already know that $B$ is continuous because it is Lipschitz). Giving scalars $\lambda_1, \ldots , \lambda_n$ and $x_1, \ldots, x_n \in X$, we have

\begin{align*} (i_Y \circ T_L)& \left(\sum_{j=1}^n \lambda_j \delta_X(x_j)\right) = \sum_{j=1}^n \lambda_ji_Y(T_L \circ \delta_X(x_j)) = \sum_{j=1}^n \lambda_j(i_Y \circ T)(x_j) \\
&= \sum_{j=1}^n \lambda_j (B \circ j_p \circ i_{\mathcal{F}(X)} \circ \delta_X)(x_j) = (B \circ j_p \circ i_{\mathcal{F}(X)})\left(\sum_{j=1}^n \lambda_j \delta_X(x_j)\right),
\end{align*}
proving that the bounded linear operators $i_Y \circ T_L$ and $B \circ j_p \circ i_{\mathcal{F}(X)}$ coincide on ${\rm span}\{\delta_X(X)\}$. But these two operators are continuous and ${\rm span}\{\delta_X(X)\}$ is dense in $\mathcal{F}(X)$, so $i_Y \circ T_L =B \circ j_p \circ i_{\mathcal{F}(X)}$. Since $j_p$ is $p$-summing it follows that $i_Y \circ T_L$ is $p$-summing as well, from which we conclude that $T_L$ is $p$-summing because the ideal of $p$-summing operators is injective. This means that the linearization of every Lipschitz $p$-summing operator is a $p$-summing linear operator. This contradicts \cite[Remark 3.3]{saadi} and completes the proof.
\end{proof}

\medskip

\begin{itemize}
		\item {\bf $(D,p)$-summing linear operators.}

	The class of $\left( D,p \right) $-summing linear operators was introduced by Mart\'{i}nez-Gim\'{e}nez and S\'{a}nchez-P\'{e}rez in \cite{MGSP}.

	\begin{definition} \cite[Definition 3.10]{MGSP}
		Let $Y$ be a Banach space and $X$ be a Banach function space compatible with the countably additive vector measure $\lambda$ of range dual pair $D = \left( \lambda, \lambda' \right) .$ A linear operator $T\colon X \longrightarrow Y$ is $\left( D, p \right) $-summing, $1 \leq p < \infty$, if there is a constant $C$ such that
				\[ \left( \displaystyle\sum_{j=1}^{m} \left\Vert T \left( f_{j} \right) \right\Vert ^{p} \right) ^{\frac{1}{p}} \leq C \sup_{g \in B_{L_{1}\left( \lambda ' \right) }} \left( \displaystyle \sum_{j=1}^{m} \left\vert \left\langle  \int_{\Omega}f_{i}\, d \lambda, \int_{\Omega'} g \, d \lambda' \right\rangle \right\vert ^{p} \right) ^{\frac{1}{p}} , \]
		for every natural $m$ and functions $f_{1},\ldots , f_{m} \in X.$
	\end{definition}

In \cite{MGSP} it is proved that $B_{L_1(\lambda')}$ is a (bounded) subset of a dual space, so its weak* closure in this dual space, denoted by $ \overline{B_{L_{1}\left( \lambda ' \right) }}$, is a compact Hausdorff space. We can consider, in our construction, $K = \overline{B_{L_{1}\left( \lambda ' \right) }}$ and the map
\begin{equation*}
	\Psi \colon  X \longrightarrow C\left( \overline{B_{L_{1}\left( \lambda ' \right) }}\right), \; \Psi(f)(g) =   \left\langle \int_{\Omega} f d \lambda, \int_{\Omega'} g d \lambda' \right\rangle.
	\end{equation*}

Thus, a linear operator $T \colon X \longrightarrow Y$ is $\left( D, p \right) $-summing if and only if $T$ is $\Psi$-Lipschitz $p$-summing. Theorem \ref{tf} characterizes $(D,p)$-summing operators by means of the following commutative diagram
\begin{equation*}
	\begin{tikzcd} C (\overline{B_{L_{1}\left( \lambda ' \right) }}) \arrow{rr}{j_p} & & L_p (\overline{B_{L_{1}\left( \lambda ' \right) }},\mu) \arrow{d}{B} \\ X \arrow[swap]{u}{\Psi} \arrow[swap]{r}{T} & Y \arrow[swap]{r}{i_{Y}} & \ell_{\infty}\left( B_{Y^*}\right) \end{tikzcd} \label{CCC}%
	\end{equation*}
where $B$ is a linear operator. Note that this characterization recovers, in an equivalent form, the original factorization theorem for this class \cite[Theorem 3.13]{MGSP}.

	\medskip

\end{itemize}

\begin{itemize}

\item {\bf Absolutely $p$-summing $\Sigma$-operators.}

In this subsection we follow the recent approach of Angulo-L\'opez and Fern\'andez-Unzueta \cite{mexicana}. Given Banach spaces $X_{1}, \ldots, X_{n}$,
$$\Sigma_{X_{1}, \ldots, X_{n}}:=\{x_{1} \otimes \cdots \otimes x_{n}\in X_{1} \otimes \cdots \otimes X_{n}: x_{1}\in X_{i} , i=1,\ldots,n\}$$ is the metric space of decomposable tensors endowed with the metric induced by the projective tensor norm. It is called the \textit{metric Segre cone} of $X_{1}, \ldots, X_{n}$. By $\mathcal{L}(\Sigma_{X_{1}, \ldots, X_{n}})$ we denote the space of scalar-valued continuous $\Sigma$-operators endowed with the Lipschitz norm, which happens to be a dual Banach space.

\begin{definition}
	Let  $X_{1}, \ldots, X_{n}, Y$ be Banach spaces. A bounded $\Sigma$-operator $f\colon
	\Sigma_{X_{1}, \dots, X_{n}}\longrightarrow Y$ is \textit{absolutely $p$-summing}, $1 \leq p < \infty$, if there
	is a $C\geq0$ so that%
	\begin{equation*}
	\left(\sum \limits_{j=1}^{m}\left\Vert f(u_{j})-f(v_{j})\right\Vert ^{p}\right) ^{\frac{1}{p}}\leq
	C\cdot\sup_{\varphi\in B_{\mathcal{L}(\Sigma_{X_{1}, \ldots, X_{n}})}}\left(\sum
	\limits_{j=1}^{m}\left\vert \varphi(u_{j})-\varphi(v_{j})\right\vert ^{p}\right) ^{\frac{1}{p}}
	\end{equation*}
	for every natural number $m$ and all $u_{j},v_{j} \in \Sigma_{X_{1}, \ldots, X_{n}}$.
\end{definition}

Choosing, in the framework developed in this paper,
\begin{equation*}
X=\Sigma_{X_{1}, \ldots, X_{n}},~K=(B_{\mathcal{L}(\Sigma_{X_{1}, \ldots, X_{n}})},w^{*}) {\rm ~and}
\end{equation*}%
\begin{equation*}
\Psi\colon \Sigma_{X_{1}, \ldots, X_{n}}\longrightarrow C(B_{\mathcal{L}(\Sigma_{X_{1}, \ldots, X_{n}})};\mathbb{K})~,~\Psi(u)(\varphi) = \varphi(u),
\end{equation*}%
we have that a bounded $\Sigma$-operator $f\colon
\Sigma_{X_{1}, \ldots, X_{n}}\longrightarrow Y$ is absolutely $p$-summing if and only if it is $\Psi$-Lipschitz $p$-summing.

Applying Theorem \ref{tf} we recover exactly the Pietsch-type factorization theorem \cite[Theorem 2.2]{mexicana}.


	\medskip
	
\item {\bf Lipschitz $p$-dominated operators.}

This class of operators was introduced by Chen and Zheng \cite{ChZh}.

\begin{definition}
	A Lipschitz mapping $T\colon X \longrightarrow Y$ between Banach spaces is {\it Lipschitz $p$-dominated}, $1 \leq p < \infty$, if there exist a Banach space $Z$ and an absolutely $p$-summing linear operator $L \colon X \longrightarrow Z$ such that
		\begin{equation*}
	 \left\Vert T(x)-T(y) \right \Vert \leq \left \Vert L(x)-L(y) \right\Vert {\rm ~for~all~} x, y \in X,
	\end{equation*}
or, equivalently (see \cite{ChZh}), if there exists a constant $C>0$ such that
	\begin{equation*}
	\left( \displaystyle \sum_{i=1}^{n} \left\Vert Tx_{i} - Ty_{i} \right\Vert^{p} \right)^{\frac{1}{p}} \leq C \sup_{x^{\ast}\in B_{X^{\ast}}} \left( \displaystyle \sum_{i=1}^{n} \left\vert x^{\ast}(x_{i} - y_{i}) \right\vert^{p}  \right) ^{\frac{1}{p}} ,
	\end{equation*}
for all $n \in \mathbb{N}$, $x_{1}, \ldots, x_{n}, y_{1}, \ldots, y_{n} \in X $.
\end{definition}
Selecting $K = B_{X^{\ast}}$ with the weak* topology and
\begin{equation}\label{bla}
\Psi \colon X \longrightarrow C\left( B_{X^{\ast}}\right)~,~\Psi(x)(x^{\ast}) =  x^{\ast}(x),
\end{equation}
it is plain that a Lipschitz mapping $T \colon X \longrightarrow Y$ is $p$-dominated if and only if $T$ is $\Psi$-Lipschitz $p$-summing.

Therefore, Theorem \ref{tf} recovers the factorization theorem for this class of mappings \cite[Theorem 3.3]{ChZh}.	

\medskip

\item {\bf Strongly Lipschitz $p$-integral operators.}

Our interest in this class, which is also due to Chen and Zheng \cite{ChZh}, relies on the fact that it is defined by means of a commutative diagram similar to the ones we are working with in this paper. By $J_Y$ we denote the canonical embedding from a Banach space $Y$ into its bidual $Y^{**}$. Remember that, for a finite measure $\mu$, $I_{\infty,p} \colon L_\infty(\mu) \longrightarrow L_p(\mu)$ denotes the canonical operator.
\begin{definition}
	A Lipschitz mapping $T\colon X \longrightarrow Y$ between Banach spaces is {\it strongly Lipschitz $p$-integral}, $1 \leq p \leq \infty$, if there are a probability measure space $(\Omega, \Sigma, \mu )$, a bounded linear operator $A\colon X \longrightarrow L_{\infty}(\mu)$ and a Lipschitz mapping $B\colon L_{p}(\mu) \longrightarrow Y^{\ast \ast}$ giving rise to the following commutative diagram:
	\[
	\begin{tikzcd}
	L_\infty (\mu) \arrow{rr}{I_{\infty, p}}  & & L_p (\mu) \arrow{d}{B} \\
	X \arrow[swap]{u}{A} \arrow[swap]{r}{T} & Y \arrow[swap]{r}{J_{Y}} & Y^{\ast \ast}
	\end{tikzcd}
	\]
\end{definition}

In \cite[Theorem 3.6]{ChZh} it is proved that every strongly Lipschitz $p$-integral operator is Lipschitz $p$-dominated. So, every strongly Lipschitz $p$-integral operator is $\Psi$-Lipschitz $p$-summing for the same map $\Psi$ in (\ref{bla}). Moreover, in \cite[Corollary 3.8]{ChZh} it is proved that the classes of Lipschitz $p$-dominated (the class whose factorization theorem we have just recovered above) and strongly Lipschitz $p$-integral operators coincide when the domain is a $C(K)$ space. Next we apply our unified factorization theorem to show that the same happens if the target space is a Lindenstrauss space. 

Recall that a Lindenstrauss space is a real Banach space whose dual is isometrically isomorphic to some $L_1(\mu)$-space.

\begin{proposition} Let $X$ be a real Banach space, $Y$ be a Lindenstrauss space and $1 \leq p < \infty$. A map $T \colon X \longrightarrow Y$ is strongly Lipschitz $p$-integral if and only if $T$ is Lipschitz $p$-dominated.
\end{proposition}

\begin{proof} One direction holds in general by \cite[Theorem 3.6]{ChZh}. For the converse, let  $T\colon X \longrightarrow Y$ be a Lipschitz $p$-dominated operator. In the previous subsection we saw that $T$ is $\Psi$-Lipschitz $p$-summing for the map $\Psi$ in (\ref{bla}). Since $Y$ is a Lindenstrauss space, $Y^{**}$ an injective Banach space, hence a $1$-absolute Lipschitz retract. Considering the canonical embedding $J_Y$, Theorem \ref{tf} guarantees the existence of a Lipschitz mapping $B\colon L_{p}(\mu) \longrightarrow Y^{\ast \ast}$ such that following diagram commutes
	\[
	\begin{tikzcd}
	L_\infty (\mu) \arrow{rr}{I_{\infty, p}}  & & L_p (\mu) \arrow{d}{B} \\
	X \arrow[swap]{u}{j_{\mu,\infty}^{\Psi}} \arrow[swap]{r}{T} & Y \arrow[swap]{r}{J_Y} & Y^{\ast \ast}
	\end{tikzcd}
	\]
Since $\Psi$ is linear, it follows that $j_{\mu,\infty}^{\Psi}$ is linear as well, proving that $T$ is strongly Lipschitz $p$-integral.
\end{proof}



\medskip

	\item {\bf Arbitrary summing operators taking values in metric spaces.}

	Here we establish a factorization theorem for a quite large (new) class of summing operators.

Given Banach spaces $X_1, \ldots, X_n$, by $\mathcal{ L}(X_1, \ldots, X_n;\mathbb{K})$ we denote the space of continuous $n$-linear functionals on $X_1 \times \cdots \times X_n$ endowed with the usual sup norm.

	\begin{definition}
		Let $X_{1},\ldots,X_{n}$ be Banach spaces, $Y=(Y,d)$ be a metric space and, for $i = 1, \ldots, n$, let $Z_i$ be a non-void subset (not necessarily a linear subspace) of $X_i$. An
		arbitrary map $T\colon Z_{1}\times\cdots\times Z_{n}\longrightarrow Y$ is
		{\it absolutely $p$-summing} if there exists $C\geq0$ such that%
		\begin{equation*}
		\left(  \sum\limits_{j=1}^{m}\left(  d\left(  T(v_{j}),T(u_{j})\right)  \right)  ^{p}\right)
		^{1/p}\\
		\leq C\underset{\varphi\in B_{\mathcal{L}(X_{1},\ldots,X_{n};\mathbb{K})}}{\sup
		}\left(\sum\limits_{j=1}^{m}\mid\varphi(v_{j})- \varphi(u_{j})\mid^{p}\right)
		^{1/p}%
		\end{equation*}
		for every $m\in\mathbb{N}$ and all $v_{j}, u_{j}\in Z_{1}\times\cdots\times Z_{n}$,
		$j=1,\ldots,m.$
	\end{definition}
	Denoting by $X_{1} \hat{\otimes}_{\pi} \cdots \hat{\otimes}_{\pi} X_{n}$ the (completed) projective tensor product of $X_1, \ldots, X_n$ and choosing $K = ( B_{\left( X_{1} \hat{\otimes}_{\pi} \cdots \hat{\otimes}_{\pi} X_{n} \right)^{\ast}}, w^*)$,
	\begin{equation*}
	\Psi\colon  Z_{1} \times \cdots \times Z_{n} \longrightarrow C\left( B_{\left( X_{1} \hat{\otimes}_{\pi} \cdots \hat{\otimes}_{\pi} X_{n} \right)^* }\right)~,~ \Psi(x_1, \ldots, x_n)(\varphi) = \varphi(x_1 \otimes \cdots \otimes x_n),
	\end{equation*}
	an arbitrary mapping $T \colon Z_{1} \times \cdots \times  Z_{n} \longrightarrow Y$ is absolutely $p$-summing if and only if $T$ is $\Psi$-Lipschitz $p$-summing.

	A domination-factorization theorem for this class of  arbitrary mappings follows from Theorem \ref{tf}.
	
\end{itemize}

\end{document}